\documentclass[11pt,oneside,reqno]{amsart}

\usepackage{amssymb,amscd,stmaryrd}
\usepackage[mathcal]{eucal}
\usepackage{array,float}
\usepackage{enumitem}
\usepackage{mathtools}
\usepackage{xcolor}

\usepackage{xy}
\input xy
\xyoption{all}
\usepackage{pdflscape}
\usepackage{hyperref}
\usepackage[left=3cm, right=3cm]{geometry}

\numberwithin{equation}{section}
\numberwithin{equation}{subsection}

\newtheorem{thm}{Theorem}[section]

\newtheorem{corollary}[thm]{Corollary}

\newtheorem{definition}[thm]{Definition}
\newtheorem*{remark*}{Remark}
\newtheorem{remark}[thm]{Remark}

\newcommand{\Hom}{{\mathrm{Hom}}}

\newcommand{\tra}{{\mathrm{tra}}}
\newcommand{\res}{{\mathrm{res}}}

\newcommand{\Ho}{{\mathrm{H}}}
\newcommand{\proj}{\mathrm{Proj}}

\newcommand{\bC}{{\mathbb C}}

\newcommand{\irr}{\mathrm{Irr}}

\title[On Projective representations of special $p$-groups]{On Projective representations of 
special $p$-groups}

\author{Sumana Hatui}
\address{School of Mathematical Sciences,
	National Institute of Science Education and Research, An OCC of Homi Bhabha National Institute, Bhubaneswar 752050, Odisha, India}

\email{sumanahatui@niser.ac.in, sumana.iitg@gmail.com}

\begin{document}

\begin{abstract}
	Let $G$ be a special $p$-group. If $G$ is of rank two, or $G$ is of maximum rank with $|G^p|\leq p$, then we describe the complex irreducible projective representations of $G$. 
\end{abstract}

\maketitle 

\section{Introduction and main results}
The theory of projective representations was first studied by Schur in \cite{IS4, IS7, IS11}.  A projective representation is a homomorphism from a group into the projective general linear groups.  Such representations arise naturally in the study of ordinary representations of groups and have significant connections to several areas, such as topology, Lie theory, and mathematical physics.  
By definition, every ordinary representation of a group is projective, but the converse need not be true.  Detailed treatment of this theory can be found in  \cite{GK}. We know that the irreducible ordinary representations of abelian groups are one dimensional, which is not valid in the case of projective representations. So the difficulty of the projective representations starts here, and it is not fully understood yet for the abelian groups. For some special cases, it has been studied  in \cite{Moa}, \cite{Moa2}, \cite{RH}.

The projective representations of Dihedral groups are also well known in the literature; see ~\cite[Theorem~7.3, page 126]{GK}. For the symmetric  groups, the projective representations were initially studied by Schur himself in \cite{IS11} and continued by several authors in \cite{Moa1,  Na1, Na2, WB}.  Projective representations of discrete Heisenberg groups over cyclic rings and over rings of order $p^2$ are described in \cite{PS, hatui2022projective}.
Continuing this line of investigation,  we study the complex irreducible projective representations of special $p$-groups $G$ for the following cases:\\
(i) $G$ is of rank $2$.\\
(ii) $G$ is of maximum rank having $G^p=1$ or $G^p \cong \mathbb Z/p\mathbb Z$. \\

A group $G$ is called a \emph{special $p$-group} of rank $k$ if $G$ is a $p$-group of order $p^n$ having the commutator subgroup $G'$ is equal to the center  $Z(G)$ and $G/G'$  is elementary abelian of order $p^{n-k}$. 
Observe that in this case $G' \cong (\mathbb Z/p\mathbb Z)^k$, $G^p\subseteq G'$ and $p \leq |G'| \leq p^{\frac{1}{2}d(d-1)}$ for $d=n-k$.  
Special $p$-groups with minimum rank are the extra-special $p$-groups, and their projective representations are described in \cite[Corollary 1.5]{PS}. In this article, we consider the special $p$-groups having rank $2$ or having maximum rank.
\\

Here we fix some notations before stating our main results.
For a group $G$, $G^p$ denotes the subgroup  generated by the elements $x^p, x \in G$.
A group $G$ is said to be  \emph{capable}  if there exists a group $H$ such that $G \cong H/Z(H)$.
The epicenter of $G$ is denoted by $Z^*(G)$, which is  the smallest central subgroup of $G$ such that $G/Z^*(G)$ is capable. 
The second cohomology group $\Ho^2(G,\mathbb C^\times)$ is called the Schur multiplier of $G$.
The set $Z^2(G, \mathbb C^\times)$  consists of all $2$-cocycles  of $G$.  For an element $\alpha \in Z^2(G, \mathbb C^\times)$, the corresponding element of $\Ho^2(G, \mathbb C^\times)$ will be denoted by $[\alpha]$.
For each $2$-cocycle $\alpha$ of $G$,  $\mathrm{Irr}^\alpha(G)$ denotes the set of all inequivalent  irreducible complex $\alpha$-representations of $G$.
We define 
$$ \text{Proj}(G):=\bigcup_{\{[\alpha] \in \Ho^2(G,\mathbb C^\times)\}}\mathrm{Irr}^\alpha(G).$$ 
The set of all  inequivalent irreducible complex ordinary representations of $G$ is denoted by $\irr(G)$.
For $m \geq 1$, $ES_p(p^{2m+1})$ and  $ES_{p^2}(p^{2m+1})$ denotes the extra-special $p$-groups   of order $p^{2m+1}$ of exponent $p$ and of exponent  $p^2$ respectively. In this article, we always assume $p$ is an  odd prime.

When we write a presentation of a group, we assume all the commutators $[x, y]$ for generators $x, y$, which are not explicitly stated in the presentation, are identity.

\bigskip

\begin{definition}
	A finite group $G^*$ is called a representation group of $G$ if there is a  subgroup   $Z\subseteq (G^*)' \cap Z(G^*) $ such that $Z\cong \Ho^2(G, \mathbb C^\times)$ and $G^*/Z \cong G$.
\end{definition}
It follows by \cite[Corollary 3.3]{PS} that there is a bijective correspondence  between the sets Proj$(G)$ and $\irr(G^*)$.

\bigskip

Now we are ready to  state our main results. 
\subsection{Special $p$-groups of rank $2$} Let $G$ be a special $p$-group of rank $2$.
An explicit description of the Schur multiplier of $G$ and a classification of capable groups among $G$ are given in \cite{hatuispecial}.
Now, the following result describes the projective representations of $G$.
\begin{thm}\label{specialrank2}
	Let $G$ be a special $p$-group of  rank $2$ minimally generated by $d$ elements, $d \geq 2$. Then the following holds:
	\begin{enumerate}
		\item  Suppose $G^p=G'$ and  $G$ is not capable, then there is a bijective correspondence between  $\mathrm{Proj}(G)$  and $\proj\big((\mathbb Z/p\mathbb Z)^{d}\big)$.

			\item If $G^p=1$ with $G$ is not capable, or $G^p\cong \mathbb Z/p\mathbb Z$, then for some central subgroup $Z \subseteq Z^*(G)$ of order $p$,
		$$G/Z \cong ES_p(p^{2m+1}) \times (\mathbb Z/p\mathbb Z)^{(d-2m)} ~\text{  for some } m \geq 1.$$\\
		(i)~  If $m \geq 2$, then there is a bijective correspondence between  $\mathrm{Proj}(G)$  and $\proj((\mathbb Z/p\mathbb Z)^{d})$.
		\\
		(ii) ~  If $m =1$, then there is a bijective correspondence between  $\mathrm{Proj}(G)$ and $\irr(G^*)$, where
		\begin{eqnarray*}
			&& G^*=\langle 
			x_1,x_2, y_r, 1 \leq r \leq d-2\mid  [x_1,x_2]=z, [z,x_k]=z_k, [y_r, x_k]=w_{rk}, [y_i,y_j]=z_{ij}, \\
			&&\hspace{4cm}  x_k^p=y_r^p=z^p=1, 1\leq i <j \leq d-2,   1\leq k \leq 2
			\rangle.
		\end{eqnarray*}
\end{enumerate}
\end{thm} 
By \cite[Proposition 3]{heineken}, we have the following.
\begin{corollary}
	If $G $ is a special $p$-group of rank $2$ of order $p^n, n\geq 8$ and $G^p=G'$, then there is a bijective correspondence between  $\mathrm{Proj}(G)$  and $\proj\big((\mathbb Z/p\mathbb Z)^{d}\big)$.
\end{corollary}

\subsection{Special $p$-groups of maximum rank}
Let $G$ be a special $p$-group  of  maximum rank, minimally generated by $d$ elements  and $d \geq 3$.  The Schur multiplier of these groups are described in \cite{rai2018schur}.
We write  $G/G'=\langle x_1 G', x_2 G', \cdots, x_d G' \rangle$ and $[x_i,x_j]=y_{ij}, ~1 \leq i<j \leq d$. Then  
$$G'\cong \bigoplus_{1 \leq i < j \leq d}\langle y_{ij} \rangle \cong (\mathbb Z/p\mathbb Z )^{\frac{1}{2}d(d-1)},$$
and every element of $G$ can be written uniquely as
\[
\prod_{i=1}^d x_i^{r_i}  \prod_{1 \leq j <k \leq d} y_{jk}^{s_{jk}} ~\text{ for } 1 \leq r_i , s_{jk}< p.
\]
We consider the two cases  $G^p=1$ and $G^p\cong \mathbb Z/p\mathbb Z$ separately. 

%We also provide an example of a group of order $p^6$ as an application of this theory.

%
%
%\begin{remark}
%	It follows from  \cite[Theorem 2.7]{niromand} that the special $p$-groups of exponent $p$ having maximum rank are capable. Hence in this case we can not apply Theorem \ref{epicenter}.
%\end{remark}

\subsubsection{$G^p=1$}
In this case,  we study projective representations of $G$ by describing $2$-cocycles  and   a representation group of $G$.
\begin{thm} \label{cocycledescription,repgroup_exp p}
	Let $G$ be a special $p$-group  of maximum rank having exponent $p$, and  $G$ be minimally generated by $d$ elements, $d \geq 3 $.  Then the following holds.
\begin{enumerate}
	\item  Every $2$-cocycle of   $G$ is  cohomologous to a cocycle of the form
	\begin{eqnarray*}
		\alpha\Big(\prod_{i=1}^d x_i^{r_i}  \prod_{1 \leq j <k \leq d} y_{jk}^{s_{jk}}, \prod_{i=1}^d x_i^{r'_i}  \prod_{1 \leq j <k \leq d} y_{jk}^{s'_{jk}}\Big) &=& \prod_{1\leq i <j <k \leq d}\mu_{ijk}^{r_is'_{jk}-r_ks'_{ij}} \prod_{1\leq    i<j <k \leq d}\mu_{jik}^{r_js'_{ik}+r_ks'_{ij}} \\
		&&\prod_{1 \leq i<j\leq d }\mu_{iij}^{r_is'_{ij}}  \prod_{1 \leq i<j\leq d }\mu_{jij}^{r_js'_{ij}} 
	\end{eqnarray*} 
	for $\mu_{ijk} \in \mathbb C^\times$ such that $\mu_{ijk}^p=1$.\\

	\item The group \begin{eqnarray*} 
		H_d^*=\langle x_i,  1 \leq i \leq d \mid [x_j, x_k]=y_{jk}, [x_i, y_{jk}]=z_{ijk},  z_{rmn}=z_{nmr}z_{mnr}^{-1}, x_i^p=1, y_{jk}^p=1,\\
		1 \leq m<n<r\leq d,  
		1 \leq j < k \leq d \rangle
	\end{eqnarray*}
	is a representation group of $G$.
\end{enumerate}
\end{thm}

\subsubsection{$G^p\cong \mathbb Z/p\mathbb Z$}
Now we state our main result for the special $p$-groups of maximum rank with $G^p\cong \mathbb Z/p\mathbb Z$. 
\begin{thm} \label{cocycleGp=p,projGp=p}
	Let $G$ be a special $p$-group of maximum rank having $|G^p| = p$, and $G$ be minimally generated by d elements with $d \geq 3$.  Then the following holds:
\begin{enumerate}
	\item Every $2$-cocycle of   $G$ is  cohomologous to a cocycle of the form
	\begin{eqnarray*}
		\alpha\big(\prod_{i=1}^d x_i^{r_i}  \prod_{1 \leq j <k \leq d} y_{jk}^{s_{jk}}, \prod_{i=1}^d x_i^{r'_i}  \prod_{1 \leq j <k \leq d} y_{jk}^{s'_{jk}}\big) = \prod_{1\leq i <j <k \leq d}\mu_{ijk}^{r_is'_{jk}-r_ks'_{ij}} 	\prod_{\substack{1\leq   i< j <k \leq d \\ (i,j)\neq (1,2)}}\mu_{jik}^{r_js'_{ik}+r_ks'_{ij}} \\
		\prod_{\substack{1 \leq i<j\leq d \\ (i,j)\neq (1,2)}}\mu_{iij}^{r_is'_{ij}}  \prod_{\substack{1 \leq i<j\leq d\\(i,j)\neq (1,2)}}\mu_{jij}^{r_js'_{ij}} 
	\end{eqnarray*} for $\mu_{ijk} \in \mathbb C^\times, 1\leq i \leq d , 1 \leq j<k \leq d$ such that $\mu_{ijk}^p=1$.\\
\item 
	There is a bijective correspondence between $\proj(G)$ and $\mathrm{Irr}(K_d^*)$, where 
	\begin{eqnarray*}
		K_d^*=\langle x_i,  1 \leq i \leq d \mid [x_j, x_k]=y_{jk},  [x_i, y_{jk}]=z_{ijk},   z_{i12}=1, 
		z_{rmn}=z_{nmr}z_{mnr}^{-1}, x_i^p=1, \\
		y_{jk}^p=1, 1 \leq m<n<r\leq d,  
		1 \leq j < k \leq d \rangle.
	\end{eqnarray*}
\end{enumerate}
\end{thm}
It follows from the above results that, to study $\proj(G)$ for special $p$-groups $G$ such that $G$ is of rank two, or $G$ is of maximum rank with $|G^p|\leq p$, it is enough to understand $\proj\big((\mathbb Z/p\mathbb Z)^{d}\big)$, $\irr(G^*), \irr(H_d^*)$ and $\irr(K_d^*)$.
We refer to \cite[Theorem  1.4 and  Section 4]{PS}  for studying $\proj\big((\mathbb Z/p\mathbb Z)^{d}\big)$. For the description of $\irr(G^*), \irr(H_d^*)$ and $\irr(K_d^*)$; see  Section \ref{ordinary}.

\section{Prerequisites}
In this section, we recall and prove the results that will be used in the preceding sections.
Let 
\begin{equation*}
1\to A \to G \to G/A \to 1
\end{equation*} 
be a central extension.
Consider a section $\mu: G/A \rightarrow G$. The map  $\tra:\Hom( A,\bC^\times) \to \Ho^2(G/A, \bC^\times) $ denotes the transgression homomorphism, which is given by $f \mapsto [\tra(f)]$, where
$$\tra(f)(\overline{x},\overline{y}) = f(\mu (\overline{x})\mu(\overline{y})\mu(\overline{xy})^{-1}),\,\, \overline{x}, \overline{y} \in G/A.$$
 The inflation homomorphism, $\inf : \Ho^2(G/A, \bC^\times)   \to  \Ho^2(G, \bC^\times)$ is given by $[\alpha] \mapsto [\inf(\alpha)]$, where $\inf(\alpha)(x,y) = \alpha(xA,yA)$. 
 The restriction homomorphism  
 $\res:  \Ho^2(G, \bC^\times)  \to  \Ho^2(A, \bC^\times)$ is defined by $[\alpha] \mapsto [\res(\alpha)]$, where $\res(\alpha)(x,y) = \alpha(x,y)$, $x,y \in A$. 
 The map
  $\chi: \Ho^2(G, \bC^\times)  \to  \Hom\big(G/(G'A)  \otimes A,\mathbb C^\times\big)$ is defined by $$\chi([\alpha])(\bar{g}, a) = \alpha(g,a) \alpha(a,g)^{-1}$$
   for all $\bar{g} = gG'A\in  G/G'A$ and $a \in A$.\\

Using  Hochschild-Serre spectral sequence \cite[Theorem 2, p.~129]{HS} and using \cite[Proposition 1.1]{IM}, 
we have the following exact sequence.
\begin{eqnarray}\label{Iwahori}
&& 1 \rightarrow \Hom( A\cap G' ,\bC^\times) \xrightarrow[]{\tra} \Ho^2(G/A, \bC^\times)  \xrightarrow{\inf} \Ho^2(G, \bC^\times) \\
&&\hspace{6cm}\xrightarrow{(\res, \chi)}  \Ho^2(A, \bC^\times)  \oplus \Hom\big(G/(G'A)  \otimes A,\mathbb C^\times\big). \nonumber
\end{eqnarray}

The special $p$-groups  $G$ has nilpotency class $2$ with $G/G'$ elementary abelian. For computing the Schur multiplier of these groups, we refer \cite{Blackburn} to the readers.
Consider $G/G'$ and $G'$ as vector spaces over the field $\mathbb Z/p\mathbb Z$.
Let $X_1$ be the subspace of $G/G' \otimes G'$ generated by the elements of the form 
\[
xG' \otimes [y,z] + yG' \otimes [z,x] +zG' \otimes [x,y], ~ x,y,z \in G
\]
and $X_2$ be the subspace  of  $G/G'\otimes G'$ generated by the elements of the form
$xG' \otimes x^p, x \in G$.   Now consider the subspace $X=X_1+X_2$.   
By \cite[Lemma 2.4]{hatuispecial}, $Z \subseteq Z^*(G)$ if and only if  $G/G' \otimes Z \subseteq X$.  
We use these notations throughout the paper without further reference.
\\

\begin{thm}\label{exactsequence}Let $G$ be a $p$-group of nilpotency class $2$ with $G/G'$  elementary abelian.  Then 
	the following sequence is exact.
	\[
	1 \to \Hom(G',\mathbb C^\times)  \xrightarrow{\tra} \Ho^2(G/G', \mathbb C^\times)   \xrightarrow{\inf}   \Ho^2(G, \mathbb C^\times)   \xrightarrow{\overline{\chi}}   \Hom\big(\frac{G/G' \otimes G'}{X},\mathbb C^\times\big) \to 1,
	\]
	where $\overline{\chi}$ is defined by $\overline{\chi}([\alpha])\big((xG'\otimes y)X\big)=\alpha(x,y)\alpha(y,x)^{-1}$ for $x\in G, y \in G'$.
\end{thm}

\bigskip

\begin{proof}	
Let $F/R$ be a free presentation of $G$. Then $G'=S/R$ for $S=F'R$ and $G/G'=F/S$.  Let $M(G)$ denotes the subgroup $\frac{F'\cap R}{[F,R]}$.
		By \cite[Theorem 2.4.6]{karpilovsky}, there is  an isomorphism $j:\Hom(M(G),\mathbb C^\times)  \to \Ho^2(G,\mathbb C^\times)$ defined as follows: 
		consider the central extension
	\[
	1 \to R/[F,R] \to F/[F,R] \to F/R \to 1.
	\]
 Then the corresponding transgression map $\tra: \Hom(R/[F,R],\mathbb C^\times) \to \Ho^2(G,\bC^\times)$ is surjective.
Let $\chi \in  \mathrm{Hom}(M(G), \mathbb C^\times)$. By \cite[Lemma 2.1.6]{karpilovsky}, $\chi$  extends to a homomorphism  $R/[F,R] \to \mathbb C^\times$. 
  Since the transgression map 
$\tra: \Hom(R/[F,R], \mathbb C^\times ) \to \Ho^2(G, \mathbb C^\times)$ is surjective, so we have
 $j (\chi)=[\alpha]$ such that 
$$\alpha(xR, yR)=\tra(\chi)(xR,yR)=\chi\big(\mu(xR)\mu(yR)\mu(xyR)^{-1}\big), x, y \in F,$$ for  a section $\mu$ of $\pi $ corresponding to the central extension
$	1 \to  R/[F,R] \to  F/[F,R] \xrightarrow{\pi} F/R \to 1$.

 By \cite[Corollary 3.2.4]{karpilovsky}, we have the following exact sequence
	\[
	1 \to \frac{G/G' \otimes G'}{X}  \xrightarrow{\psi} M(G)\xrightarrow{\phi}   M(G/G')\xrightarrow{\tau} G' \to 1,
	\] 
	where $\psi$ is defined by $\psi\big((fF'R\otimes xR)X\big)=[f,x][F,R]$.
Applying hom functor, the above exact sequence induces the following exact sequence
 \[
	1 \to \Hom(G',\mathbb C^\times)  \xrightarrow{\overline{\tau} }\mathrm{Hom}(M(G/G'), \mathbb C^\times)    \xrightarrow{\overline{\phi} }\mathrm{Hom}(M(G), \mathbb C^\times)  \xrightarrow{\overline{\psi} }  \Hom(\frac{G/G' \otimes G'}{X} ,\mathbb C^\times) \to 1
	\]	
	Using \cite[Proposition 11.8.16]{LR},  we have the following exact sequence from  (\ref{Iwahori}).  
		\[
	1 \to \Hom(G',\mathbb C^\times)  \xrightarrow{\tra} \Ho^2(G/G', \mathbb C^\times)   \xrightarrow{\inf}   \Ho^2(G, \mathbb C^\times)   \xrightarrow{\chi}   \Hom\big(G/G' \otimes G',\mathbb C^\times\big).
	\]
Consider the following diagram
\begin{figure}\small
\xymatrix{ 
	1 \ar[r] &  \mathrm{Hom}(G', \mathbb C^\times) \ar[d]^{id} \ar[r]^{\overline{\tau} }& \mathrm{Hom}(M(G/G'), \mathbb C^\times)  \ar[d]^{j} \ar[r]^{\overline{\phi} } &\mathrm{Hom}(M(G), \mathbb C^\times)  \ar[r]^{\overline{\psi} }  \ar[d]^{j} &  \Hom(\frac{G/G' \otimes G'}{X} ,\mathbb C^\times) \ar[d]^{id} \ar[r] & 1
\\
	1 \ar[r] &   \mathrm{Hom}(G', \mathbb C^\times) \ar[r]^{\mathrm{tra}} & \Ho^2(G/G', \mathbb C^\times) \ar[r] ^{\mathrm{inf}} & \Ho^2(G, \mathbb C^\times) \ar[r]^{\overline{\chi} }  &  \Hom(\frac{G/G' \otimes G'}{X} ,\mathbb C^\times) \ar[r] & 1.\\
}
\end{figure}

Now we prove that the diagram is commutative. Then, the second row is  exact as the first row is exact.

Consider the following commutative diagram with the natural homomorphisms $i_1,i_2, \pi_1, \pi_2, \beta_1$.
\[
\xymatrix{ 
	1 \ar[r] & S/[F,S] \ar[d] \ar[r]^{i_1} & F/[F,S] \ar[r]^{\pi_1} \ar[d]^{\beta_1}& F/S  \ar[d]^{id} \ar[r] & 1\\
		1 \ar[r] & S/R \ar[r]^{i_2}  & F/R\ar[r]^{\pi_2}  & F/S \ar[r] & 1\\
}
\]
It is easy to see that, if $\mu_1$ is a section of $\pi_1$, then $\gamma=\beta_1\circ \mu_1$ is a section of $\pi_2$.
Let $\{y_i; 1 \leq i \leq m\}$ be a coset representatives of $F/S$. Define a section $\mu_1: F/S \to F/[F,S]$ by $\mu_1(y_iS)=y_i[F,S]$. Suppose $y_iy_jS=y_kS$.
For $h \in  \mathrm{Hom}(G', \mathbb C^\times)$, we have
$$(j \circ \overline{\tau}(h))(y_iS,y_jS)=\overline{\tau}(h) \big(\mu_1(y_iS)\mu_1(y_jS)\mu_1(y_iy_jS)^{-1}[F,S]\big)=h(y_iy_jy_k^{-1}R)$$ and $$\tra(h)(y_iS,y_jS)=h\big(\gamma(y_iS)\gamma(y_jS)\gamma(y_iy_jS)^{-1}R\big)=h(y_iy_jy_k^{-1}R).$$ 

Thus $$j \circ \overline{\tau}=\tra.$$

Consider the following commutative diagram with the natural homomorphisms $i_1,i_2, \pi_1, \pi_2, \beta_1,\beta_2$.
\[
\xymatrix{ 
	1 \ar[r] & R/[F,R] \ar[r]^{i_1} \ar[d] & F/[F,R]\ar[r]^{\pi_1} \ar[d]^{\beta_1} & F/R \ar[d]^{\beta_2} \ar[r] & 1\\
	1 \ar[r] & S/[F,S] \ar[r]^{i_2} & F/[F,S] \ar[r]^{\pi_2} & F/S  \ar[r] & 1\\
}
\]
It is easy to see that, if $\mu_2$ is a section of $\pi_1$, then $\gamma': F/S\to F/[F,S]$ defined by 
$\gamma'(xS)=\beta_1\circ \mu_2(xR)$ is a section of $\pi_2$.
Let $\{x_1,x_2, \ldots, x_n\} \in F$ be a coset representative of $F/R$ and we define a section $\mu_2$ of $\pi_1$ by $\mu_2(x_iR)=x_i[F,R]$.
Suppose $x_ix_j R=x_kR$. For $h \in  \mathrm{Hom}(M(G/G'), \mathbb C^\times) $, we have
$$j \circ \overline{\phi}(h) (x_iR,x_jR)= \overline{\phi}(h)(x_i x_j x_k^{-1}[F,R])=h(x_i x_j x_k^{-1}[F,S]),$$
and 
$$\inf \circ j (h) (x_iR,x_jR)= j (h) (x_iS,x_jS)=h\big(\gamma'(x_iS)\gamma'(x_jS)\gamma'(x_ix_jS)^{-1}[F,S]\big)=h(x_i x_j x_k^{-1}[F,S]).$$
Hence, 
$$j \circ \overline{\phi}=\inf \circ j.$$

Observe that, for $h \in \mathrm{Hom}(M(G), \mathbb C^\times)$, we have
\begin{eqnarray*}
\overline{\chi}\circ j(h)((fS \otimes xR)X)&=&j(h)(fR,xR) j(h)(xR,fR)^{-1}, ~f \in F, x\in S\\
&=&h\big(\mu_2(fR)\mu_2(xR) \mu_2(fxR)^{-1}\mu_2(xfR)\mu_2(fR)^{-1}\mu_2(xR)^{-1}\big)\\
&=&h\big(f[F,R]x[F,R]f^{-1}[F,R]x^{-1}[F,R]\big) \text { as } fxR=xfR\\
&=&h\big([f,x][F,R]\big).
\end{eqnarray*}
and $$
\overline{\psi}(h)\big((fS \otimes xR)X\big)=h\circ \psi\big((fS \otimes xR)X\big)=h\big( [f,x][F,R]\big).
$$
	\end{proof}

If a group is not capable, then the following results say that every projective representation of $G$ can be obtained from the projective representation of $G/Z^*(G)$.
\begin{thm}\label{epicenter}
	Let $G$ be a finite group  and  $Z$  be a central subgroup of $G$ such that $Z \subseteq Z^*(G)$. Then there is a bijective correspondence between the sets
$\proj(G)$  and $\proj(G/Z)$.
\end{thm}
\begin{proof}
We have the following exact sequence 
		\[
	1 \to \Hom(Z\cap G',\mathbb C^\times)  \xrightarrow{\tra} \Ho^2(G/Z, \mathbb C^\times)   \xrightarrow{\inf}   \Ho^2(G, \mathbb C^\times).
	\]
By \cite[Theorem 2.5.10]{karpilovsky}, the map $\inf$ is surjective if and only if   $Z \subseteq Z^*(G)$.
	Hence, the result follows from \cite[Theorem 3.2]{PS}.
\end{proof}

\begin{remark}
	In \cite{HKYp5}, we provide a classification of capable groups of order $p^5$, and we describe $Z^*(G)$ for each non-capable group $G$. To study projective representations of non-capable groups $G$ of order $p^5$, it is enough to study projective representations of $G/Z^*(G)$, follows from the above result.
\end{remark}

\begin{thm}\label{capabilityG^p=p}
	Let $G$ be a special $p$-group with $|G^p|=p$, and $G$ be minimally generated by $d$ elements, $d \geq 3$. Then $G$ is not capable, and there is a bijective correspondence between the sets
	$
	\proj(G) 
	$ 
	and $\proj(G/G^p)$.
\end{thm}
\begin{proof}
	By \cite[Lemma 2.4]{hatuispecial}, we have  that $G^p \subseteq Z^*(G)$ as $G/G'\otimes G^p \subseteq X$. Hence, the result follows from Theorem \ref{epicenter}.
	\end{proof}

\begin{thm}\label{maxrankSchurmultiplier}
	Let $G$ be a special $p$-group of maximum rank. 
	Then
	\[
	\Ho^2(G, \mathbb C^\times)   \cong \Hom\Big(\frac{G/G' \otimes G'}{X},\mathbb C^\times\Big).
	\]
	In particular, if $G$ is minimally generated by $d$ elements, then
	\[ \Ho^2(G, \mathbb C^\times)   \cong \begin{cases} 
	(\mathbb Z/p\mathbb Z)^{\frac{1}{3}d(d-1)(d+1)},  & G^p=1, \\
	(\mathbb Z/p\mathbb Z)^{\frac{1}{3}d(d-1)(d+1)-d},  & |G^p|=p. \\
	\end{cases}
	\]
\end{thm}
\begin{proof} 
		Since  $G/G'\cong (\mathbb Z/p\mathbb Z)^{d}$ and $G'\cong (\mathbb Z/p\mathbb Z)^{\frac{1}{2}d(d-1)}$, $\Ho^2(G/G', \mathbb C^\times) \cong 
		\Hom(G' ,\mathbb C^\times)$. Hence, result follows from Theorem \ref{exactsequence} and \cite[Corollary 2.10]{Niroomand}. 
		\end{proof}

		In Theorem \ref{maxrankSchurmultiplier}, 
		$ \overline{\chi}: \Ho^2(G, \mathbb C^\times)   \to \Hom\big(\frac{G/G' \otimes G'}{X},\mathbb C^\times\big)$
		is an isomorphism, and the inverse map		
		 $\eta: \Hom\Big(\frac{G/G' \otimes G'}{X},\mathbb C^\times\Big) \to \Ho^2(G, \mathbb C^\times)$ is given as follows: 
		let $\{g_1, g_2, \ldots, g_n\}$ be a transversal of $G'$ in $G$.
		Then any element of $G$ is of the form $g_ih$ for $h \in G'$. 
Consider the map  $\eta$   given by 
 $\eta(f)=[\alpha]$ such that 
$$\alpha(g_ih_1,g_jh_2 )=f((g_iG' \otimes h_2)X), h_1,h_2 \in G'.$$  
Since $ \overline{\chi} \circ \eta(f)=f$ and  $\overline{\chi}$ is an isomorphism,  $\eta$ is an inverse map of $\bar{\chi}$.

\section{Proofs of the main results}
In this section, we prove our main results.

\begin{proof}[\textbf{Proof of Theorem \ref{specialrank2}}]
		
\emph{(1)} Suppose $G^p=G'$. By \cite[Theorem 1.1]{hatuispecial}, for every central subgroup $Z\subseteq Z^*(G)$ of order $p$, 
		$$G/Z \cong ES_{p^2}(p^{2m+1}) \times (\mathbb Z/p\mathbb Z)^{(d-2m)}\text{ for some }m \geq 1.$$
		Suppose $H$ denotes the group $ES_{p^2}(p^{2m+1}) \times (\mathbb Z/p\mathbb Z)^{(d-2m)}$. Observe that  $H/H' \cong (\mathbb Z/p\mathbb Z)^{d}$.
		Using \cite[Theorem 2.2.10, Theorem 3.3.6]{karpilovsky}, we get
		$\Ho^2(H, \mathbb C^\times) \cong \frac{\Ho^2( H/H', \mathbb C^\times)}{\Hom(H',\mathbb C^\times)}$, and hence,
		the map 
		$$\inf: \Ho^2( H/H', \mathbb C^\times)\to  \Ho^2(H, \mathbb C^\times)$$ is surjective.
		Thus by \cite[Theorem 3.2]{PS}, 
		there is a bijective correspondence between   $\proj(H)$ and $\proj( (\mathbb Z/p\mathbb Z)^{d})$. Now the result follows from Theorem \ref{epicenter}.\\

\emph{(2)}	It follows from \cite[Theorem 1.3, Theorem 1.4]{hatuispecial} that   
	$$G/Z \cong ES_p(p^{2m+1}) \times (\mathbb Z/p\mathbb Z)^{(d-2m)}, ~m \geq 1,$$
	for some central subgroup $Z\subseteq Z^*(G)$ of order $p$.
	Suppose $H$ denotes the group $ES_p(p^{2m+1}) \times (\mathbb Z/p\mathbb Z)^{(d-2m)}$. \\	
	
	$(i)$  Observe that  $H/H' \cong (\mathbb Z/p\mathbb Z)^{d}$.  For  $m\geq 2$,
	the map 
	$$\inf: \Ho^2( H/H', \mathbb C^\times)\to  \Ho^2(H, \mathbb C^\times)$$ is surjective. By \cite[Theorem 3.2]{PS}, it follows that
	there is a bijective correspondence between   $\proj(H)$ and $\proj( (\mathbb Z/p\mathbb Z)^{d})$.  Now result follows from Theorem \ref{epicenter}.\\
	
	$(ii)$ Observe that 
	\begin{eqnarray}\label{eq1}
		G^* &\cong &\big( \langle x_1,z, y_r, 1 \leq r \leq d-2\mid [z,x_1]=z_1, [y_r, x_1]=w_{r1}, [y_i,y_j]=z_{ij}, x_k^p=y_r^p=z^p=1,\nonumber\\ 
		&&  \hspace{4cm}  1\leq i <j \leq d-2\rangle   \times \langle  z_2\rangle \times \langle w_{r2}, 1 \leq r \leq d-2 \rangle\big) \rtimes \langle x_2\rangle\nonumber\\
		&\cong&  \Big(\big(\langle y_r, 1 \leq r \leq d-2\mid  [y_i,y_j]=z_{ij}, 1\leq i <j \leq d-2 \rangle \times  \langle w_{rk}, 1 \leq r \leq d-2, k=1,2 \rangle\nonumber\\
		&&  \hspace{6cm}  \langle z,z_1,z_2\rangle   \big)  \rtimes \langle x_1\rangle    \Big) \rtimes \langle x_2\rangle\nonumber\\
		&\cong& \Big(\big(K \times (\mathbb Z/p\mathbb Z)^{2(d-2)+3}\big) \rtimes  \mathbb Z/p\mathbb Z\Big) \rtimes  \mathbb Z/p\mathbb Z,
	\end{eqnarray}	
where 
\[
K=\langle y_r, 1 \leq r \leq d-2\mid  [y_i,y_j]=z_{ij}, 1\leq i <j \leq d-2 \rangle.
\]
	From \cite[Theorem 1.4]{PS}, it follows that $K$ is a representation group of the abelian group $(\mathbb Z/p\mathbb Z)^{d-2}$. Hence $K$ has order $p^{\frac{1}{2}(d-2)(d-3)+(d-2)}$. 
	So $G^*$ is of order 
	$$p^{\frac{1}{2}(d-2)(d-3)+(d-2)+(2d+1)}=p^{\frac{1}{2}d(d+1)+2}.$$ 
Using \cite[Theorem 2.2.10, Theorem 3.3.6]{karpilovsky}, we have 
	$$\Ho^2(G/Z,\mathbb C^\times)\cong (\mathbb Z/p\mathbb Z)^{\frac{1}{2}d(d-1)+1}.$$
	Observe that $G/Z$ is of order $p^{d+1}$ 
	 and $$\Ho^2(G/Z,\mathbb C^\times)\cong \langle z_1,z_2, w_{r1}, w_{r2}, z_{ij}, 1 \leq r \leq d-2, 1 \leq i <j \leq d-2  \rangle.$$	 
	 Therefore, 
	 $$G^*/\langle z_1, z_2, w_{r1}, w_{r2}, z_{ij} , 1 \leq r \leq d-2, 1 \leq i <j \leq d-2  \rangle \cong G/Z,$$ and $G^*$ is a representation group of $G/Z$.  By Theorem \cite[Corollary 3.3]{PS}, there is a bijection between the sets $\proj(G/Z)$ and $\irr(G^*)$.
	Now result follows from Theorem \ref{epicenter}.
	\end{proof}

\begin{proof}[\textbf{Proof of Theorem \ref{cocycledescription,repgroup_exp p}}]
By \cite[Lemma 2.1]{Niroomand},  $G$ has the presentation
\[
\langle x_r, 1\leq r \leq d  \mid [x_i,x_j]=y_{ij}, x_r^p=y_{ij}^p=1,~1 \leq i<j \leq d\rangle.
\]
Every element  $g \in G$ can be written uniquely as 
\[
g=\prod_{i=1}^d x_i^{r_i}  \prod_{1 \leq j <k \leq d} y_{jk}^{s_{jk}} \text{ for } 0 \leq r_i , s_{jk} <p. 
\]
\emph{(1)}
It follows from Theorem \ref{maxrankSchurmultiplier} that, upto cohomologous every $2$-cocycle $\alpha$ is of the form  
$$\alpha\big(\prod_{i=1}^d x_i^{r_i}  \prod_{1 \leq j <k \leq d} y_{jk}^{s_{jk}}, \prod_{i=1}^d x_i^{r'_i}  \prod_{1 \leq j <k \leq d} y_{jk}^{s'_{jk}}\big) =f\Big( \big(\prod_{i=1}^d x_i^{r_i} G' \otimes \prod_{1 \leq j <k \leq d} y_{jk}^{s'_{jk}}\big)X\Big),$$ 
for a homomorphism $f \in \Hom(\frac{G/G'\otimes G'}{X}, \mathbb C^\times)$. 
We have  $|X|=p^{d \choose 3}$ and $$X =\langle  x_iG'\otimes y_{jk}+x_jG'\otimes y_{ki}+x_kG'\otimes y_{ij}, 1 \leq i<j<k\leq d\rangle$$
Observe that $$f\big((x_kG'\otimes y_{ij}) X\big)=f\big((x_iG'\otimes y_{jk})X\big)^{-1}f\big((x_jG'\otimes y_{ik}) X\big)~ \forall~ 1 \leq i<j<k\leq d,$$
and hence, we have
\begin{eqnarray*}
	&&f\Big(\big(\prod_{i=1}^d x_i^{r_i}G'\otimes \prod_{1 \leq j <k \leq d} y_{jk}^{s'_{jk}}\big)X\Big)=\prod_{\substack{ 1 \leq i \leq d
			\\ 1 \leq j <k \leq d}} f\big( (x_i G'\otimes  y_{jk})X\big)^{r_is'_{jk}}\\
		&=&\prod_{1\leq i <j <k \leq d} f\big(( x_i G'\otimes  y_{jk})X\big)^{r_is'_{jk}}\prod_{1\leq i <j <k \leq d} f\big( (x_j G' \otimes y_{ik})X\big)^{r_js'_{ik}} \\
		&&  \prod_{1\leq i <j <k \leq d} f\big(( x_kG'\otimes y_{ij})X\big)^{r_ks'_{ij}}
		 \prod_{1 \leq i<j\leq d }f\big(( x_iG'\otimes   y_{ij})X\big)^{r_is'_{ij}}  \\
		 &&\prod_{1 \leq i<j\leq d }f\big(( x_j G'\otimes  y_{ij})X\big)^{r_js'_{ij}} \\
		&=&\prod_{1\leq i <j <k \leq d} f\big(( x_i G'\otimes   y_{jk})X\big)^{r_is'_{jk}-r_ks'_{ij}}\prod_{1\leq i <j <k \leq d} f\big(( x_j G'\otimes   y_{ik})X\big)^{r_js'_{ik}+r_ks'_{ij}} \\
		&&	  \prod_{1 \leq i<j\leq d }f\big(( x_i G' \otimes   y_{ij})X\big)^{r_is'_{ij}}  
	\prod_{1 \leq i<j\leq d }f\big(( x_j G'\otimes   y_{ij})X\big)^{r_js'_{ij}} 
	\end{eqnarray*}
Assume  that $f\big(( x_i G'\otimes y_{jk})X\big)=\mu_{ijk} \in \mathbb C^\times, 1\leq i \leq d , 1 \leq j<k \leq d$. We have $f\big((x_i^pG'\otimes y_{jk})X\big)=\mu_{ijk}^p=1$. Hence the result follows.\\

\emph{(2)} 
Observe that $|G|=p^{\frac{1}{2}d(d+1)}$ and  by Theorem \ref{maxrankSchurmultiplier}, $$\Ho^2(G, \mathbb C^\times)\cong (\mathbb Z/p\mathbb Z)^{\frac{1}{3}d(d-1)(d+1)}.$$
First we show that $H_d^*$ has order $|G||\Ho^2(G,\mathbb C^\times )|=p^{\frac{1}{2}d(d+1)+\frac{1}{3}d(d-1)(d+1)}=p^{\frac{1}{6}d(d+1)(2d+1)}$.
We prove it by induction on $d$. 
	For $d=2$,  
	$H_2^*$ has the following presentation 
	$$\langle x_1,x_2 \mid [x_1,x_2]=y_{12}, [x_1, y_{12}]=z_{112}, [x_2, y_{12}]=z_{212}, x_1^p=x_2^p= y_{12}^p=1 \rangle $$ which is a group of order $p^5$. Hence the result is true for $d=2$. Now 
\begin{eqnarray}\label{eq2}
H_d^*\cong  && \Big(\big\langle x_i,  1 \leq i \leq d-1 \mid [x_j, x_k]=y_{jk}, [x_i, y_{jk}]=z_{ijk},     z_{rmn}=z_{nmr}z_{mnr}^{-1},  x_i^p=1,\nonumber\\
&& 1 \leq j < k \leq d-1,   1 \leq m<n<r\leq d-1 \big\rangle \times \langle y_{id}\rangle _{i=1}^{d-1}  \times \langle z_{iid}\rangle _{i=1}^{d-1}   \times \langle z_{did}\rangle _{i=1}^{d-1} \nonumber \\
 && \hspace{4cm} \times \langle z_{mnd}\rangle _{1 \leq m <n \leq d-1}\times \langle z_{nmd}\rangle _{1 \leq m <n \leq d-1}  \Big) \rtimes \langle  x_d \rangle \nonumber \\
&& \cong \big(H_{d-1}^* \times  (\mathbb Z/p\mathbb Z)^{3(d-1)+2.\frac{1}{2}(d-1)(d-2)}\big) \rtimes \mathbb Z/p\mathbb Z,
\end{eqnarray}
where 
	\begin{eqnarray*}
H_{d-1}^*=\langle x_i,  1 \leq i \leq d-1 \mid [x_j, x_k]=y_{jk}, [x_i, y_{jk}]=z_{ijk},     z_{rmn}=z_{nmr}z_{mnr}^{-1},  x_i^p=1,   \\
1 \leq m<n<r\leq d-1, 1 \leq j < k \leq d-1 \rangle.
	\end{eqnarray*}
By induction, $|H_{d-1}^*|=p^{\frac{1}{6}(d-1)d(2d-1)}$.
Hence $H_d^*$ has order 
$$p^{\frac{1}{6}(d-1)d(2d-1)+3(d-1)+(d-1)(d-2)+1}=p^{\frac{1}{6}d(d+1)(2d+1)}.$$
Consider the subgroup $$K=\langle {z_{mnr}}\rangle _{1\leq m < n <r \leq d} \times \langle z_{nmr} \rangle _{1\leq m < n <r \leq d} \times \langle z_{mmn} \rangle _{1\leq m < n \leq d} \times \langle z_{nmn} \rangle _{1\leq m < n \leq d}$$ of $G$. 
Then  $K \cong \Ho^2(G,\mathbb C^\times )$ and  $H_d^*/K \cong G$. Hence the result follows.
	\end{proof}
	
%Now we consider the following example to see how this theory will help us to write the $2$-cocycles of $G$ explicitly and  to write a representation group.

%\begin{example}
%	Consider the group $G=\langle x_1,x_2,x_3 \mid [x_1,x_2]=y_{12},   [x_2,x_3]=y_{23},  [x_1,x_3]=y_{13}, x_i^p=y_{ij}^p=1\rangle$ of order $p^6$.
%	G is a group of order $p^6$ with $G^p=1$.
%\end{example}
%%Upto cohomologous the $2$-cocycles of $G$ are of the form
%%\[
%%\alpha(\prod_{i=1}^{3})
%%\]
%%Then representation groups of $G$ is
%%\[
%%G^*=\langle   x_1,x_2,x_3 \mid   [x_1,x_2]=y_{12},   [x_2,x_3]=y_{23},  [x_1,x_3]=y_{13}, [x_i, y_{jk}]=z_{ijk},  x_i^p=y_{ij}^p=1\rangle.
%%\]
%%

\begin{proof}[\textbf{Proof of Theorem  \ref{cocycleGp=p,projGp=p}}]
	By \cite[Theorem 2.13, Corollary 2.14]{Niroomand}, $G$ has the presentation
	\[
	\langle x_r, 1\leq r \leq d  \mid [x_i,x_j]=y_{ij}, y_{12}=x_1^p, x_1^{p^2}=x_j^p=y_{ij}^p=1,~1 \leq i<j \leq d\rangle.
	\]
	\emph{(1)} Here we have  $$X =\langle  x_iG'\otimes y_{jk}+x_jG'\otimes y_{ki}+x_kG'\otimes y_{ij},  ~x_r\otimes y_{12}, 1 \leq i<j<k\leq d, 1\leq r \leq d\rangle.$$
	Now proof goes on the same lines as proof of Theorem \ref{cocycledescription,repgroup_exp p}.\\

\emph{(2)} The group $G/G^p$ has the following presentation
	\[
G/G^p \cong 	\langle x_r, 1\leq r \leq d  \mid [x_i,x_j]=y_{ij}, y_{12}=1, x_1^{p}=x_j^p=y_{ij}^p=1,~1 \leq i<j \leq d\rangle.
	\]
% $H^*$ has order $|G/G^p||\Ho^2(G/G^p,\mathbb C^\times)|=p^{\frac{1}{2}d(d+1)-1+\frac{1}{3}d(d-1)(d+1)-d+1}=p^{\frac{1}{6}d(d+1)(2d+1)-d}.$
Since 
\begin{eqnarray}  \label{eq10}
K_d^* \cong H_d^*/\langle z_{i12} \rangle_{1 \leq i \leq d},
\end{eqnarray}
 $K_d^*$ is  of order
$p^{\frac{1}{6}d(d+1)(2d+1)-d}.$
Now consider the subgroup 
	\begin{eqnarray*}
		L=\langle y_{12}\rangle \times \big\langle {z_{mnr}}\big\rangle _{1\leq m < n <r \leq d} \times \big\langle z_{nmr} \big\rangle _{\substack{1\leq m < n <r \leq d\\ (m,n)\neq (1,2)}} \times \big\langle z_{mmn} \big\rangle _{\substack{1\leq m < n \leq d \\ (m,n)\neq (1,2)}}
		 \times \big\langle z_{nmn} \big\rangle_{\substack{1\leq m < n \leq d \\ (m,n)\neq (1,2)}} 
			\end{eqnarray*}
			of $K_d^*$. Since  
			$$\Ho^2(G/G^p, \mathbb C^\times) \cong (\mathbb Z/p\mathbb Z)^{\frac{1}{3}d(d-1)(d+1)-d+1} \cong L$$ and  $K_d^*/L \cong G/G^p$, $K_d^*$ is a representation group of $G/G^p$. 
	 Now result follows from Theorem \ref{capabilityG^p=p}.
	\end{proof}

\section{Ordinary representations of $G^*, H_d^*, K_d^*$ for $d \geq 3$}\label{ordinary}
It follows by the previous discussions that, to complete our study, it is enough to understand   the inequivalent irreducible  complex ordinary representations of the following groups.
	\begin{eqnarray*}
	 G^*&=&\langle 
	x_1,x_2, y_r, 1 \leq r \leq d-2\mid  [x_1,x_2]=z, [z,x_k]=z_k, [y_r, x_k]=w_{rk}, [y_i,y_j]=z_{ij}, \\
	&&\hspace{5cm}  x_k^p=y_r^p=z^p=1, 1\leq i <j \leq d-2,   1\leq k \leq 2
	\rangle.\\
	H_d^*&=&\langle x_i,  1 \leq i \leq d \mid [x_j, x_k]=y_{jk}, [x_i, y_{jk}]=z_{ijk},  z_{rmn}=z_{nmr}z_{mnr}^{-1}, x_i^p=y_{jk}^p=1, \\
 && \hspace{7cm}	1 \leq m<n<r\leq d,  
	1 \leq j < k \leq d \rangle.\\
		K_d^*&=&\langle x_i,  1 \leq i \leq d \mid [x_j, x_k]=y_{jk},  [x_i, y_{jk}]=z_{ijk},   z_{i12}=1, 
	z_{rmn}=z_{nmr}z_{mnr}^{-1}, \\
	&& \hspace{5cm}	x_i^p=y_{jk}^p=1,
1 \leq m<n<r\leq d,  
	1 \leq j < k \leq d \rangle.
\end{eqnarray*}
The results given in \cite[Section 4]{PS}  explain a method for a construction of ordinary representations, which helps us to describe ordinary representations of these groups.

\subsection{$\irr(G^*)$}
By \eqref{eq1}, 
\begin{eqnarray*}
	G^*
	&\cong & \Big(\big(K \times (\mathbb Z/p\mathbb Z)^{2(d-2)+3}\big) \rtimes  \mathbb Z/p\mathbb Z\Big) \rtimes  \mathbb Z/p\mathbb Z.
\end{eqnarray*}
We study $\irr(G^*)$  as follows: Suppose $N^*=\Big(\big(K \times (\mathbb Z/p\mathbb Z)^{2(d-2)+3}\big) \rtimes  \mathbb Z/p\mathbb Z\Big) $.
Then $N^*$ is a normal subgroup of $G^*$ such that $G^*/N^*$ is a cyclic group of order $p$. Hence, we determine the inertia group of the elements of $\irr(N^*)$ in $G^*$, and then the construction is obtained by \cite[Proposition 4.2]{PS}.

We study $\irr(N^*)$  as follows:    $N^*$ has a normal subgroup $N=\big(K \times (\mathbb Z/p\mathbb Z)^{2(d-2)+3}\big)$  such that $N^*/N$ is a cyclic group of order $p$. Since $N$ is of nilpotency class $2$, the irreducible representations of $N$ are obtained from one dimensional representation of the radical of each central character as described in \cite[Section 4.1]{PS}. 
Then it remains to determine the inertia group of these representations of $N$ in $N^*$, and then the construction is obtained by \cite[Proposition 4.2]{PS}.

%We note that  is a normal subgroup of $N^*$ such that $N^*/N$ is cyclic. This helps to understand $\irr(N^*)$.

\subsection{$\irr(H_d^*)$ and $\irr(K_d^*)$}
By \eqref{eq2} and \eqref{eq10},
	\[
	H_d^* 
	 \cong  \big(H_{d-1}^* \times  (\mathbb Z/p\mathbb Z)^{d^2-1}\big) \rtimes \mathbb Z/p\mathbb Z \text{ and }
	K_d^*=\frac{H_d^*}{\langle {z_{i12}} \rangle_{1 \leq i \leq d}}.
	\]
Observe that,	
\[
K_d^* \cong  \big(K_{d-1}^*  \times  (\mathbb Z/p\mathbb Z)^{d^2-2}\big) \rtimes \mathbb Z/p\mathbb Z.
\]
We use inductive process  on $d$ to study $\irr(H_d^*)$ and $\irr(K_d^*)$ for $d \geq 3$. 
For $d=3$,  
	\begin{eqnarray*}
		H_3^*&=&\langle x_i, 1 \leq i \leq 3 \mid [x_j,x_k]=y_{jk}, [x_i, y_{jk}]=z_{ijk}, z_{321}=z_{123}z_{213}^{-1}, \\
	&& \hspace{7cm} x_i^p= y_{jk}^p=1, 1\leq j <k\leq 3 \rangle.\\
	&\cong & \big(\langle x_1,x_2 \mid [x_1,x_2]=y_{12}, [x_1, y_{12}]=z_{112}, [x_2, y_{12}]=z_{212}, x_i^p= y_{jk}^p=1 \rangle\\
	&& \hspace{6cm} \times  (\mathbb Z/p\mathbb Z)^{8}\big) \rtimes \mathbb Z/p\mathbb Z.\\
	K_3^*&=&H_3^*/\langle z_{112}, z_{212}, z_{312}\rangle.\\
	&\cong &  \big(\langle x_1,x_2 \mid [x_1,x_2]=y_{12}, x_1^p=x_2^p= y_{12}^p=1 \rangle \times  (\mathbb Z/p\mathbb Z)^{7}\big) \rtimes \mathbb Z/p\mathbb Z.
		\end{eqnarray*}
		Now consider the  Heisenberg group over $\mathbb Z/p\mathbb Z$, say $H_3(\mathbb Z/p\mathbb Z)$, and $\hat{H}(p,1)$ is a representation group of $H_3(\mathbb Z/p\mathbb Z)$; see \cite[Theorem 1.2]{PS}. Then we have
\[
H_3^* \cong \big(\hat{H}(p,1)  \times  (\mathbb Z/p\mathbb Z)^{8}\big) \rtimes \mathbb Z/p\mathbb Z.
\]
\[
K_3^* \cong \big(H_3(\mathbb Z/p\mathbb Z) \times  (\mathbb Z/p\mathbb Z)^{7}\big) \rtimes \mathbb Z/p\mathbb Z.
\] 
It is easy to see that the groups $H_3^*$ and $K_3^*$ have order $p^{14}$ and $p^{11}$ respectively. 
The groups $H_3^*$ and $K_3^*$ have normal subgroups  $\hat{H}(p,1)  \times  (\mathbb Z/p\mathbb Z)^{8}$ and $H_3(\mathbb Z/p\mathbb Z) \times  (\mathbb Z/p\mathbb Z)^{7}$ respectively such that the corresponding quotients are cyclic groups of order $p$. Hence, the ordinary representations of these groups can be obtained using the theory described in \cite[Section 4]{PS}.

Now assume $d >3$. The groups $H_d^*$ and $K_d^*$ have normal subgroups $N_H=\big(H_{d-1}^* \times  (\mathbb Z/p\mathbb Z)^{d^2-1}\big)$ and $N_K=\big(K_{d-1}^* \times  (\mathbb Z/p\mathbb Z)^{d^2-2}\big)$  respectively. Since  $H_d^*/N_H$ and  $K_d^*/N_K$ are cyclic groups of order $p$, it remains to determine the inertia group of  the elements of $\irr(N_H)$ and $\irr(N_K)$ in $H_d^* $ and $K_d^*$,  respectively. Then, the construction is obtained by \cite[Proposition 4.2]{PS}.

\section*{Acknowledgements} 

The author  thanks  the National Institute of Science Education and Research, Bhubaneswar for providing an excellent research facility.

\bibliographystyle{amsplain}
\bibliography{projective}

\providecommand{\bysame}{\leavevmode\hbox to3em{\hrulefill}\thinspace}
\providecommand{\MR}{\relax\ifhmode\unskip\space\fi MR }
\providecommand{\MRhref}[2]{%
  \href{http://www.ams.org/mathscinet-getitem?mr=#1}{#2}
}
\providecommand{\href}[2]{#2}
\begin{thebibliography}{10}

\bibitem{Blackburn}
Evens~L. Blackburn, N., \emph{Schur multipliers of p-groups}, Journal für die
  reine und angewandte Mathematik \textbf{309} (1979), 100--113.

\bibitem{hatuispecial}
Sumana Hatui, \emph{Schur multipliers of special $p$-groups of rank $2$},
  Journal of Group Theory \textbf{23} (2020), no.~1, 85--95.

\bibitem{HKYp5}
Sumana Hatui, Vipul Kakkar, and Manoj~K Yadav, \emph{The schur multiplier of
  groups of order $p^5$}, Journal of Group Theory \textbf{22} (2019), no.~4,
  647--687.

\bibitem{hatui2022projective}
Sumana Hatui, E.~K. Narayanan, and Pooja Singla, \emph{Projective
  representations of \uppercase{H}eisenberg groups over the rings of order
  $p^2$}, arXiv preprint arXiv:2202.07243 (2022).

\bibitem{PS}
Sumana Hatui and Pooja Singla, \emph{On schur multiplier and projective
  representations of \uppercase{H}eisenberg groups}, Journal of Pure and
  Applied Algebra \textbf{225} (2021), no.~11, 106742.

\bibitem{heineken}
Hermann Heineken, \emph{Nilpotent groups of class two that can appear as
  central quotient groups}, Rendiconti del Seminario Matematico della
  Universit{\`a} di Padova \textbf{84} (1990), 241--248.

\bibitem{RH}
R.~J Higgs, \emph{Projective representations of abelian groups}, Journal of
  Algebra \textbf{242} (2001), no.~2, 769--781.

\bibitem{HS}
Gerhard Hochschild and Jean-Pierre Serre, \emph{Cohomology of group
  extensions}, Transactions of the American Mathematical Society (1953),
  110--134.

\bibitem{IM}
N.~Iwahori and H.~Matsumoto, \emph{Several remarks on projective
  representations of finite groups},  \textbf{10} (1964), 129--146.

\bibitem{Niroomand}
Farangis Johari and Peyman Niroomand, \emph{A note on some special $p$-groups},
  arXiv preprint arXiv:1807.04959 (2018).

\bibitem{GK}
Gregory Karpilovsky, \emph{Projective representations of finite groups}, New
  York-Basel (1985).

\bibitem{karpilovsky}
\bysame, \emph{The schur multiplier}, Oxford University Press, Inc., 1987.

\bibitem{Moa1}
A.~O. Morris, \emph{The spin representation of the symmetric group}, Canadian
  Journal of Mathematics \textbf{17} (1965), 543--549.

\bibitem{Moa}
\bysame, \emph{Projective representations of abelian groups}, Journal of the
  London Mathematical Society \textbf{2} (1973), no.~2, 235--238.

\bibitem{Moa2}
A.~O. Morris, M.~Saeed-ul Islam, and E.~Thomas, \emph{Some projective
  representations of finite abelian groups}, Glasgow Mathematical Journal
  \textbf{29} (1987), no.~2, 197--203.

\bibitem{Na1}
M.~L. Nazarov, \emph{An orthogonal basis in irreducible projective
  representations of the symmetric group}, Funktsional. Anal. i Prilozhen
  \textbf{22} (1988), no.~1, 77--78.

\bibitem{Na2}
\bysame, \emph{Young's orthogonal form of irreducible projective
  representations of the symmetric group}, Journal of the London Mathematical
  Society \textbf{2} (1990), no.~3, 437--451.

\bibitem{rai2018schur}
Pradeep~K. Rai, \emph{On the schur multiplier of special p-groups}, Journal of
  Pure and Applied Algebra \textbf{222} (2018), no.~2, 316--322.

\bibitem{IS4}
J.~Schur, \emph{{\"U}ber die darstellung der endlichen gruppen durch gebrochen
  lineare substitutionen.}, Journal f{\"u}r die reine und angewandte Mathematik
  \textbf{1904} (1904), no.~127, 20--50.

\bibitem{IS7}
\bysame, \emph{Untersuchungen {\"u}ber die darstellung der endlichen gruppen
  durch gebrochene lineare substitutionen.}, Journal f{\"u}r die reine und
  angewandte Mathematik \textbf{1907} (1907), no.~132, 85--137.

\bibitem{IS11}
\bysame, \emph{{\"U}ber die darstellung der symmetrischen und der
  alternierenden gruppe durch gebrochene lineare substitutionen.}, Journal
  f{\"u}r die reine und angewandte Mathematik \textbf{1911} (1911), no.~139,
  155--250.

\bibitem{LR}
L.~R. Vermani, \emph{An elementary approach to homological algebra}, CRC Press,
  2003.

\bibitem{WB}
David~B Wales, \emph{Some projective representations of $s_n$}, Journal of
  algebra \textbf{61} (1979), no.~1, 37--57.

\end{thebibliography}

\end{document}